\documentclass{amsart}
\usepackage{amssymb}
\usepackage{mathrsfs}
\usepackage{amsmath}
\usepackage{amsthm}
\usepackage{bbm}
\usepackage{times}
\usepackage{datetime}
\usepackage{scrtime}
\usepackage{cases}
\usepackage{amsfonts}
\usepackage{graphicx}
\usepackage{stmaryrd}
\usepackage{amssymb}
\usepackage{mathrsfs}
\usepackage{amsthm}
\usepackage{amsmath}
\usepackage{bbm}

\newtheorem{theorem}{Theorem}[section]
\newtheorem{corollary}[theorem]{Corollary}
\newtheorem{lemma}[theorem]{Lemma}
\newtheorem{proposition}[theorem]{Proposition}
\newtheorem{definition}{Definition}

\numberwithin{equation}{section}

\newtheorem{problem}[theorem]{Problem}

\newtheorem*{claim}{Claim}
\newtheorem*{thma}{Almansi's Theorem}

\newenvironment{thmbis}[1]
  {\renewcommand{\thetheorem}{\ref{#1}$'$}%
   \addtocounter{theorem}{-1}%
   \begin{theorem}}
  {\end{theorem}}

\newcommand{\hyperg}[4]{\: _2\! F_1 \left( #1,\, #2;\, #3;\, #4 \right)}

\newcommand{\RePt}{\mathrm{Re}\,}

\newcommand{\ball}{\mathbb{B}}
\newcommand{\sphere}{\mathbb{S}}
\newcommand{\disk}{\mathbb{D}}

\newcommand{\bfI}{\mathbf{I}}
\newcommand{\bfL}{\mathbf{L}}
\newcommand{\bfP}{\mathbf{P}}
\newcommand{\bfM}{\mathbf{M}}
\newcommand{\bfR}{\mathbf{R}}

\newcommand{\PH}{\mathrm{PH}}

\newcommand{\bbC}{\mathbb{C}}

\newcommand{\bbR}{\mathbb{R}}

\newcommand{\bbN}{\mathbb{N}}
\newcommand{\bbS}{\mathbb{S}}

\begin{document}

\title[Weighted integrability of polyharmonic functions]{Weighted integrability of polyharmonic functions in the higher dimensional case}%

\thanks{The first author was supported by the National Natural Science Foundation of China grant 11971453. The second author acknowledges financial support from the Spanish Ministry of Economy and Competitiveness, through the Mar\'ia de Maeztu Programme for Units of Excellence in R\&D (MDM-2014-0445), Academy of Finland project no. 268009, and the grant MTM2017-83499-P (Ministerio de Educaci\'on y Ciencia).}

\author{Congwen Liu}
\email{cwliu@ustc.edu.cn}

\address{CAS Wu Wen-Tsun Key Laboratory of Mathematics,
School of Mathematical Sciences,
University of Science and Technology of China\\
Hefei, Anhui 230026,
People's Republic of China}

\author{Antti Per\"al\"a}
\email{perala.math@gmail.com}

\address{Department of Mathematical Sciences,
Chalmers University of Technology and the University of Gothenburg,
Gothenburg SE-412 96, Sweden}

\author{Jiajia Si}
\email{sijiajia@mail.ustc.edu.cn}

\address{School of Science, Hainan University, Haikou, Hainan 570228,
People's Republic of China.}

\keywords{Polyharmonic functions, Weighted integrability, Boundary behaviour, Cellular decomposition}%

\subjclass[2010]{31B30, 35J40}

\begin{abstract}
This paper is concerned with the $L^p$ integrability of $N$-harmonic functions with respect to the standard weights $(1-|x|^2)^{\alpha}$ on the unit ball $\mathbb{B}$ of $\mathbb{R}^n$, $n\geq 2$. More precisely, our goal is to determine the real (negative) parameters $\alpha$, for which $(1-|x|^2)^{\alpha/p} u(x) \in L^p(\mathbb{B})$ implies that $u\equiv 0$, whenever $u$ is a solution of the $N$-Laplace equation on $\mathbb{B}$. This question is motivated by the uniqueness considerations of the Dirichlet problem for the $N$-Laplacian $\Delta^N$.

Our study is inspired by a recent work of Borichev and Hedenmalm \cite{BH14}, where a complete answer to the above question in the case $n=2$ is given for the full scale $0<p<\infty$. When $n\geq 3$, we obtain an analogous characterization for $\frac{n-2}{n-1}\leq p<\infty$, and remark that the remaining case can be genuinely more difficult.
Also, we extend the remarkable cellular decomposition theorem of Borichev and Hedenmalm to all dimensions.
\end{abstract}

\maketitle

\section{Introduction}

A complex-valued function $u$ defined on a bounded domain $\Omega$ in the Euclidean
space $\bbR^n$ is polyharmonic of order $N$ (or $N$-harmonic) if $u$ is $2N$ times continuously differentiable and
\[
\Delta^{N} u(x) = 0 \quad \text{for all } x\in \Omega,
\]
where $\Delta^{N}$ is the $N$-th iterate of the Laplacian
\[
\Delta := \frac {\partial^2}{\partial x_1^2} + \cdots + \frac {\partial^2}{\partial x_n^2}.
\]
A polyharmonic function of order $1$ is just a harmonic function; for $N=2$, the term
biharmonic function, which is important in elasticity theory, is used.
There is a vast literature on polyharmonic functions, see \cite{ACL83} and \cite{GGS10} for basic references.

We denote by $\textsc{PH}_{N}(\Omega)$ the linear space of all $N$-harmonic functions on $\Omega$.
Also, we let $L_{\alpha}^{p}(\Omega)$ be the space of measurable functions $f:\Omega\to \bbC$ with
\[
\|f\|_{p,\alpha} := \int\limits_{\Omega} |f(x)|^p [\mathrm{dist}(x,\partial \Omega)]^{\alpha} dV(x) < \infty,
\]
where $dV$ is the Lebesgue measure on $\mathbb{R}^n$.
We put
\[
\textsc{PH}_{N,\alpha}^{p}(\Omega):=\textsc{PH}_{N}(\Omega)\cap L_{\alpha}^{p}(\Omega),
\]
and endow it with the norm or quasi-norm structure of $L_{\alpha}^{p}(\Omega)$. This is obviously the subspace of
$L_{\alpha}^{p}(\Omega)$ consisting of $N$-harmonic functions.

In their remarkable paper \cite{BH14}, Borichev and Hedenmalm raised the following question.
\begin{problem}\label{prob:1.1}
For which triples $(N,p,\alpha)$ do we have that $\PH_{N,\alpha}^{p}(\Omega) = \{0\}$?
\end{problem}

The interesting case is when $\alpha$ is negative. Then the integrability asks for
the function to decay in mean at some rate along the boundary. This is closely related to
the uniqueness issues associated with the Dirichlet problem
for the $N$-Laplacian equation
\begin{equation}\label{eqn:Diri}
\begin{cases}
\Delta^N u = 0 \quad \text{in }\Omega,\\
\partial_{\mathrm{n}}^j u = f_j \quad \text{on } \partial \Omega \text{ for } j=0,1,\ldots, N-1,
\end{cases}
\end{equation}
where $\partial_{\mathrm{n}}$ stands for the (interior) normal derivative.
See \cite[Subsection 1.3]{BH14} for a detailed background.

There clearly exists a critical number $\beta(N,p)$ such that
\[
\PH_{N,\alpha}^{p}(\Omega) = \{0\} \quad \text{for} \quad \alpha < \beta(N,p)\\
\]
and
\[
\PH_{N,\alpha}^{p}(\Omega) \neq \{0\} \quad \text{for} \quad \alpha > \beta(N,p).
\]
In fact, $\beta(N,p)$ can be given explicitly by
\[
\beta(N,p):= \inf\{ \beta_p(u): u\in \mathrm{PH}_N (\Omega)\setminus \{0\} \},
\]
where for a Borel measurable function $u:\Omega\to \bbC$,
\[
\beta_p(u) ~:=~ \inf \{ \alpha\in \bbR: u\in L_{\alpha}^p(\Omega) \}.
\]
If $u\not\in L_{\alpha}^p(\Omega)$ for every $\alpha \in \bbR$, we write $\beta_p(u) := +\infty$.
Following \cite{BH14}, we call the function $p\mapsto \beta(N,p)$ the critical integrability type curve for the
$N$-harmonic functions, and the function $(N,p)\mapsto \beta(N,p)$ the critical integrability type
curves for the polyharmonic functions.

When $n=2$ and $\Omega$ is the unit disk $\disk$ in the plane, Borichev and Hedenmalm \cite{BH14} completely resolved
Problem \ref{prob:1.1} by giving an explicit formula for $\beta(N,p)$, the critical integrability type
curves for the polyharmonic functions. To avoid repetition, we do not include the detailed results here.

The aim of this paper is to extend the main results of \cite{BH14} to all
dimensions. Let $\ball$ stand for the open unit ball of $\bbR^n$. Also, we write $\sphere$ for the unit sphere, the boundary of $\ball$. By $d\sigma$, we mean the $(n-1)$-dimensional surface measure on $\sphere$, normalized so that $\sigma(\sphere)=1$. We investigate the Problem \ref{prob:1.1} when $\Omega=\ball$,
for $n\geq 2$.
Our first main result is the following:
\begin{theorem}\label{thm:main2}
The critical integrability type curve for the polyharmonic
functions on $\ball$ is given by
\begin{equation}\label{eqn:crit-type}
\beta(N,p) = \min_{j:0\leq j\leq N} b_{j,N}(p)
\end{equation}
for $N\in \bbN$ and $\frac{n-2}{n-1}\leq p <\infty$, where
\begin{align}
b_{0,N}(p) ~:=~& -1-(N-1)p, \\
b_{j,N}(p) ~:=~& \max\{-1-(N+j-1)p, -n-(N-j-n+1)p\} \label{eqn:bjNp}
\end{align}
for $j = 1,\ldots,N$. In particular, when $n\geq 3$,
\begin{equation}\label{eqn:crit-type2}
\beta(N,p) ~=~ \begin{cases}
-1-Np, & \text{if }\,  \frac{n-2}{n-1} \leq p < \frac{n-1}{n}, \\
-n-(N-n)p, & \text{if }\,  \frac{n-1}{n} \leq  p < 1, \\
-1-(N-1)p, & \text{if }\,  p\geq 1.
\end{cases}
\end{equation}
Here and throughout this paper, when $n=2$, the expression $\frac{n-2}{n-1}\leq p <\infty$ should be interpreted as $0<p<\infty$.
\end{theorem}

The requirement $p\geq \frac{n-2}{n-1}$ stems from the subharmonicity of the gradient (see \cite{SW60}), which is a non-issue when $n=2$. It is natural to expect that the formula \eqref{eqn:crit-type} in Theorem \ref{thm:main2} is true for the
full range of $p$: $0<p<\infty$. Unfortunately, this is not the case if $n\geq 3$. See Section 8 for an explanation of why.

A novel decomposition theorem of polyharmonic functions on the unit disk,
referred as to the cellular decomposition theorem, played a crucial role in the analysis performed in
\cite{BH14}. It is closely related to the following

\begin{thma}
If $u$ is polyharmonic of order $N$ on $\ball$, then
there exist unique harmonic functions $u_0,\ldots,u_{N-1}$ on $\ball$ such that
\begin{equation}\label{eqn:almansiexpansion}
u(x) = u_{0}(x) + |x|^2 u_{1}(x) + \cdots + |x|^{2N-2} u_{N-1}(x), \quad x\in \ball.
\end{equation}
\end{thma}
%In fact \eqref{eqn:almansiexpansion} remains valid when $\ball$ is replaced by a star domain with center $0$, but
%we only use the case of a ball.

Almansi's theorem plays a key role in the theory of polyharmonic functions.
See \cite{ACL83, HK93, Kou92, Nic35}, to mention only a few. However, it is not good enough for the
analysis of the weighted integrability of polyharmonic functions, for the individual terms in the
alternative Almansi expansion (which follows by rearranging the terms in \eqref{eqn:almansiexpansion},
see \cite[p.500, (2)]{Pav97})
\begin{equation}\label{eqn:yaalmansi}
u(x) = v_{0}(x) + (1-|x|^2)v_{1}(x) + \cdots + (1-|x|^2)^{N-1} v_{N-1}(x),
\end{equation}
where all $v_j$ are harmonic, do not necessarily have the same decay rates near the boundary.
See Subsection 2.5 and Section 3 in \cite{BH14} for the detailed explanations.

To remedy this situation, Borichev and Hedenmalm \cite{BH14} established
a modified Almansi representation (\cite[Theorem 3.4]{BH14}):
\begin{equation} \label{eqn:celluardecomp0}
u=w_{0}+ M [w_{1}]+ \cdots +M^{N-1}[w_{N-1}],
\end{equation}
where $M[u](z) := (1 - |z|^2) u(z)$ and each function $w_j$ solves
the differential equation $L_{N-j-1}[w_{j}]=0$, with
\begin{equation}\label{eqn:Ltheta0}
L_{\theta} [u](z) ~:=~ (1-|z|^2) \Delta u(z) + 4 \theta\, [z \partial_z u(z) + \bar{z} \bar{\partial}_z u(z)]
- 4\theta^2 u(z).
\end{equation}
It is a crucial feature of this expansion that each term in \eqref{eqn:celluardecomp0} remains in the
space $\mathrm{PH}_{N,\alpha}^{p}(\disk)$, whenever $u \in \mathrm{PH}_{N,\alpha}^{p}(\disk)$.
This means that one may analyze each term separately.
%However, here the terms are mixed in way that is optimal for the boundary behaviour.
%Our second result is a higher-dimensional generalization of this decomposition.

%Almansi's formula plays a key role in the theory of polyharmonic functions.

%Roughly speaking, Almansi's theorem represents a polyharmonic function in $\ball$ as a polynomial
%of degree $N-1$ in $|x|^2$ with coefficients that are harmonic in $\ball$.
%We rewrite the Almansi expansion \eqref{eqn:almansiexpansion} as
%\begin{equation}\label{eqn:almansi}
%u(x) = v_{0}(x) + (1-|x|^2)v_{1}(x) + \cdots + (1-|x|^2)^{N-1} v_{N-1}(x), \quad x\in \ball,
%\end{equation}
%where $v_j$ are given by
%\[
%v_j := (-1)^{j} \sum_{k=j}^{N-1} \binom{k}{j} u_k, \quad j=0,\ldots, N-1,
%\]
%which are harmonic functions on $\ball$.
%Thus, Almansi's theorem states that an $N$-harmonic function on $\ball$ can be represented
%as a polynomial of degree $N-1$ in $1-|x|^2$ with the coefficient functions $v_j$ harmonic in $\ball$.

%The Almansi expansion (2.2) can be expressed in the following form:
%Conversely it is evident by differentiation that such a polynomial is m-harmonic. Since

We shall extend the cellular decomposition theorem in \cite{BH14} to higher dimensions.
To this end, the first and key step is to find the higher dimensional analogue
of the differential operator $L_{\theta}$.
This turns out to be
\begin{equation}\label{eqn:Ltheta}
\bfL_{\theta} [u](x) ~:=~ (1-|x|^2) \Delta u(x) + 4 \theta\, \bfR [u](x) + 2\theta (n-2-2\theta) u(x),
\end{equation}
where $\theta$ is a real parameter and $\bfR [u] (x) := x \cdot \nabla f(x)$ is the radial derivative of $u$.
It has appeared implicitly in \cite{Leu87,LP04,LP09} and is closely related to the theory of axially
symmetric potentials developed by Weinstein (see, e.g., \cite{Wei53}).
%Also, when $n=2$, the operator $\mathbf{L}_{\theta}$ is related in a simple way to the operators $D_{\alpha}$ introduced by
%Olofsson in \cite{Olo14}. See \cite[p. 474]{BH14}.
%Our second result establishes a representation formula for polyharmonic function analogous to \eqref{eqn:almansi},
%but with coefficients $v_j$, $j = 0,..., N-1$, being solutions of a family of differential operators given by
As we will see later, this differential operator enjoys exactly the same properties
as those for $L_{\theta}$, especially the operator identities \eqref{eqn:commu1} and
\eqref{eqn:commu3}, which enable us to establish the following

\begin{theorem}[The modified Almansi representation]\label{thm:mdfalmansi}
For every polyharmonic function $u$ of order $N$ on $\ball$, there exist unique
functions $w_0, \ldots, w_{N-1}$, satisfying $\bfL_{N-j-1}[w_{j}]=0$ for $j=0,\ldots, N-1$, such that
\begin{equation}\label{eqn:modifiedalmansi}
u(x) = w_{0}(x) + (1-|x|^2) w_{1}(x) + \cdots + (1-|x|^2)^{N-1} w_{N-1}(x), \quad x\in \ball.
\end{equation}
%where the functions $w_{j}$ are $(N-j)$-harmonic and solve  $\bfL_{N-j-1}[w_{j}]=0$ on $\ball$,
%for $j=0,\ldots,N-1$. Here $\mathbf{L}_{\theta}$ is
%the second order elliptic partial differential operator given by
%\begin{equation}\label{eqn:Ltheta}
%\bfL_{\theta} [u] ~:=~ (1-|x|^2) \Delta u + 4 \theta \bfR [u] + 2\theta (n-2-2\theta) u,
%\end{equation}
%where $\theta$ is a real parameter and $\bfR [u] (x) := x \cdot \nabla f(x)$ is the radial derivative of $u$.
\end{theorem}
Roughly speaking, the alternative Almansi representation \eqref{eqn:yaalmansi} states that an $N$-harmonic function on $\ball$ can be represented
as a polynomial of degree $N-1$ in $1-|x|^2$ with all the coefficient functions $v_j$ harmonic on $\ball$.
Our Theorem \ref{thm:mdfalmansi} represents an $N$-harmonic function by a similar formula,
but each coefficient $w_j$ here solves its own partial differential equation
$\bfL_{N-j-1}[w_{j}]=0$.

Furthermore, we have the following

\begin{theorem}[The cellular decomposition theorem]\label{thm:cellular}
Let $0<p<\infty$, $N\in \bbN$ and $\alpha\in \bbR$. Then every
$u\in\mathrm{PH}_{N,\alpha}^{p}(\ball)$ has a unique decomposition
\begin{equation} \label{eqn:celluardecomp}
u = \sum_{j=0}^{N-1} \bfM^j [w_{j}],
\end{equation}
with each term $\bfM^{j}[w_{j}]$ being in $\mathrm{PH}_{N,\alpha}^{p}(\ball)$, where the function $w_{j}$
is $(N-j)$-harmonic and solves  $\bfL_{N-j-1}[w_{j}]=0$ on $\ball$.
Here, for notational convenience, we write
\[
\bfM^j [w](x) := (1-|x|^2)^j w(x), \quad x\in \ball.
\]
%Here $\mathbf{L}_{\theta}$ is
%the second order elliptic partial differential operator given by
%\begin{equation}\label{eqn:Ltheta}
%\bfL_{\theta} [u] ~:=~ (1-|x|^2) \Delta u + 4 \theta \bfR [u] + 2\theta (n-2-2\theta) u,
%\end{equation}
%where $\theta$ is a real parameter and $\bfR [u] (x) := x \cdot \nabla f(x)$ is the radial derivative of $u$.
\end{theorem}

Theorem \ref{thm:cellular} can be improved by specifying which terms in the decomposition \eqref{eqn:celluardecomp}
must necessarily vanish.

Following \cite{BH14}, we denote by $\mathcal{A}_N$ the open set
\begin{equation}
\mathcal{A}_N ~:=~ \left\{ (p,\alpha)\in \bbR^2: 0<p<+\infty \text{ and } \alpha> \beta(N,p)\right\}
\end{equation}
for fixed $N\geq 2$, and refer to it as the admissible region. So the definition of $\beta(N,p)$
is equivalent to the statement
\[
(p,\alpha) \in \mathcal{A}_N \quad \Longleftrightarrow \quad \PH_{N,\alpha}^p (\ball) \neq \{0\}.
\]
Denote by $\widetilde{\mathcal{A}}_N$ the subset of $\mathcal{A}_N$:
\begin{equation}
\widetilde{\mathcal{A}}_N := \left\{ (p,\alpha)\in \bbR^2: \tfrac {n-2}{n-1}\leq p<+\infty \text{ and }
\alpha> \min_{j:0\leq j \leq N} b_{j,N}(p)\right\}.
\end{equation}
For a point $(p,\alpha)\in \mathcal{A}_N$, we put
\[
J(p,\alpha) := \{ j \in \{0,\ldots,N-1\} : \alpha > a_{N-j,N}(p)\},
\]
where
\begin{equation}
a_{j,N}(p) ~:=~ \min\{b_{j,N}(p),-1-(N-j)p\}
\end{equation}
for $N\in \mathbb{N}$ and $j\in \{1,\cdots,N\}$.

\begin{theorem}\label{thm:main3}
Suppose $(p,\alpha)\in \widetilde{\mathcal{A}}_N$.
Then every $u\in\mathrm{PH}_{N,\alpha}^{p}(\ball)$ has a unique decomposition
\begin{equation}
u = \sum_{j\in J(p,\alpha)} \bfM^{j}[w_{j}],
\end{equation}
where each term $\bfM^{j}[w_{j}]$ is in $\mathrm{PH}_{N,\alpha}^{p}(\ball)$, while the functions $w_{j}$ are $(N-j)$-harmonic and solve  $\bfL_{N-j-1}[w_{j}]=0$ on $\ball$, for $j\in J(p,\alpha)$.
\end{theorem}

Note that each term $\bfM^{j}[w_{j}]$ with $j\in J(p,\alpha)$ is allowed to be nontrivial, so the above result is sharp.

We follow the strategy of \cite{BH14} whenever applicable. There are some notable differences,
such as the lack of powerful tools from complex analysis that only work in the plane. In addition,
instead of defining $N$-harmonicity in the sense of distribution theory, we can use the more
elementary standard definition (but our results remain valid in the former case).
This is the case, because we use the simpler test functions in Lemma \ref{lem:Phi_theta},
without resorting to method of Olofsson \cite{Olo14}.

The rest of the paper is organized as follows:
Section 2 is devoted to the basic properties of the differential operator $\bfL_{\theta}$.
Our main results, Theorems \ref{thm:mdfalmansi}, \ref{thm:cellular} and \ref{thm:main2} will be proved
in Section 3, Section 4 and Sections 5-7, respectively.
Theorem \ref{thm:main3} is then proved in Section 8.
The last Section 9 is devoted to concluding remarks and open problems.

\subsubsection*{Acknowledgement}
This work began while the first author was visiting the Department of Mathematics and Statistics,
University of Helsinki. He wishes to express his
gratitude for the warm hospitality he received there, especially from Professors
Mats Gyllenberg, Tuomas Hyt\"onen, Pertti Mattila, Jari Taskinen and
Hans-Olav Tylli.
We are also grateful to Professors Kehe Zhu and Guangbin Ren for many helpful discussions and comments.
Finally, we would like to thank the referee for his/her constructive suggestions and comments.
In particular, we were advised to develop the modified Almansi representation for an
arbitrary polyharmonic function, which led to Theorem 1.3.

\section{The differential operator $\bfL_{\theta}$}

\subsection{Some elementary identities}

Let $\lambda$ be a real number.  We define the multiplication operator $\bfM^{\lambda}$ by
\[
\bfM^{\lambda} [u] (x) := (1-|x|^2)^{\lambda} u(x),\quad x\in \ball,
\]
and in particular, $\bfM:=\bfM^{1}$. We also write $\bfM^{0}:=\bfI$.

The following proposition is called the correspondence principle.

\begin{proposition}\label{prop:corrprinciple}
For any $\theta, \lambda\in \bbR$, we have
\begin{equation}\label{eqn:corrp}
\bfL_{\theta} \bfM^{\lambda} = \bfM^{\lambda} \bfL_{\theta-\lambda} + 4\lambda(\lambda-1-2\theta) \bfM^{\lambda-1}.
\end{equation}
\end{proposition}

\begin{proof}
We first compute
\begin{align*}
\Delta \left\{\left(1-|x|^2\right)^{\lambda}  u(x)\right\} ~=~& (1-|x|^2)^{\lambda} \Delta u(x) + 2 \nabla \left\{(1-|x|^2)^{\lambda}\right\}\cdot \nabla u(x) \\
& \quad + u(x)\Delta \left\{(1-|x|^2)^{\lambda}\right\} \\
=~& (1-|x|^2)^{\lambda} \Delta u(x)  - 4 \lambda (1-|x|^2)^{\lambda-1} \bfR u(x) \\
& \quad -2\lambda(2\lambda+n-2) (1-|x|^2)^{\lambda-1}  u(x) \\
& \quad  + 4\lambda(\lambda-1) (1-|x|^2)^{\lambda-2}  u(x),
\end{align*}
which can be written as
\begin{equation}\label{eqn:DeltaM}
\Delta \bfM^{\lambda}  = \bfM^{\lambda} \Delta - 4 \lambda\bfM^{\lambda-1} \bfR
-2\lambda(2\lambda +n-2) \bfM^{\lambda-1} + 4\lambda(\lambda-1) \bfM^{\lambda-2}.
\end{equation}
Also, it is easy to verify that
\[
\bfR \bfM^{\lambda} = \bfM^{\lambda} \bfR  + 2\lambda \bfM^{\lambda} - 2\lambda \bfM^{\lambda-1}.
\]
Therefore,
\begin{align*}
\bfL_{\theta} \bfM^{\lambda} ~=~&  \bfM \Delta \bfM^{\lambda}  + 4\theta \bfR \bfM^{\lambda} + 2\theta (n-2-2\theta) \bfM^{\lambda}\\
=~& \bfM \left\{\bfM^{\lambda} \Delta - 4 \lambda\bfM^{\lambda-1} \bfR  -2\lambda(2\lambda+n-2) \bfM^{\lambda-1}
+ 4\lambda(\lambda-1) \bfM^{\lambda-2} \right\} \\
& \quad + 4\theta (\bfM^{\lambda} \bfR  + 2\lambda \bfM^{\lambda} - 2\lambda \bfM^{\lambda-1}) + 2\theta (n-2-2\theta) \bfM^{\lambda}\\
=~& \bfM^{\lambda} \left\{\bfM \Delta + 4(\theta-\lambda) \bfR +  2(\theta-\lambda) (n-2-2\theta+2\lambda) \bfI \right\} \\
&\quad + 4\lambda(\lambda-1-2\theta) \bfM^{\lambda-1} \\
=~& \bfM^{\lambda} \bfL_{\theta-\lambda} + 4\lambda(\lambda-1-2\theta) \bfM^{\lambda-1},
\end{align*}
as desired.

\end{proof}

We single out two special cases of Proposition \ref{prop:corrprinciple} as separate statements.

\begin{corollary}
For any $\theta \in \bbR$ we have
\begin{equation}\label{eqn:commu1}
\bfL_{\theta} \bfM  = \bfM \bfL_{\theta-1}  - 8 \theta \bfI.
\end{equation}
More generally,
\begin{equation}\label{eqn:commu2}
\bfL_{\theta} \bfM^j = \bfM^{j} \bfL_{\theta-j} + 4j(j-1-2\theta) \bfM^{j-1}, \quad j=1,2,\ldots.
\end{equation}
\end{corollary}

\begin{corollary}
For any $\theta \in \bbR$ we have
\begin{equation}\label{eqn:commu5}
\bfL_{\theta} \bfM^{1+2\theta }   = \bfM^{1+2\theta} \bfL_{-\theta-1}.
\end{equation}
\end{corollary}

\begin{proposition}
We have that
\begin{equation}\label{eqn:commu3}
\Delta \bfL_{\theta} = \bfL_{\theta-1} \Delta.
\end{equation}
More generally,
\begin{equation}\label{eqn:commu4}
\Delta^j \bfL_{\theta} = \bfL_{\theta-j} \Delta^j, \quad j=1,2,\ldots.
\end{equation}
\end{proposition}

\begin{proof}
It is clear that
\[
\Delta \bfR = \bfR \Delta  + 2 \Delta .
\]
Also, by applying \eqref{eqn:DeltaM} to $\Delta u$, we get
\[
\Delta \bfM \Delta = \bfM \Delta^2 - 4 \bfR \Delta - 2n \Delta .
\]
It follows that
\begin{align*}
\Delta \bfL_{\theta}  ~=~&  \Delta \bfM \Delta + 4\theta \Delta \bfR + 2\theta (n-2-2\theta) \Delta\\
=~& (\bfM \Delta^2 - 4 \bfR \Delta -2n \Delta)  + 4\theta (\bfR \Delta + 2 \Delta)  + 2\theta (n-2-2\theta) \Delta\\
=~& \big\{\bfM \Delta + 4(\theta-1) \bfR +  2(\theta-1) (n-2\theta) \bfI \big\} \Delta \\
=~& \bfL_{\theta-1} \Delta.
\end{align*}
The identity \eqref{eqn:commu4} follows by iteration of \eqref{eqn:commu3}.
\end{proof}

The next result for $n=2$ is Proposition 6.1 from \cite{BH14}. It explains the usefulness of the operators $\bfL_{\theta}$.

\begin{proposition}
We have the following factorization:
\begin{equation}\label{eqn:factorization}
\bfL_0 \bfL_1 \cdots \bfL_{N-1} = \bfM^N \Delta^N , \quad N = 1,2,3,\ldots
\end{equation}
\end{proposition}

\begin{proof}
Since, by definition, $\bfL_0=\bfM \Delta$, the assertion holds trivially for $N = 1$.
Suppose now that it holds for $N = k$;
\[
\bfL_0 \bfL_1 \cdots \bfL_{k-1} = \bfM^{k} \Delta^{k}.
\]
Then, by \eqref{eqn:commu4},
\[
\bfL_0 \bfL_1 \cdots \bfL_{k} = \bfM^{k} \Delta^{k} \bfL_{k} = \bfM^{k} \bfL_0 \Delta^{k} = \bfM^{k} (\bfM \Delta) \Delta^{k}
= \bfM^{k+1} \Delta^{k+1}.
\]
The proof is completed by virtue of the induction principle.
\end{proof}

\begin{corollary}\label{cor:LthetaandNharmonic}
If $u$ solves $\bfL_{N-1} [u] = 0$ in $\ball$, then $u$ is $N$-harmonic in $\ball$.
More generally, if $u$ solves $\bfL_{N-j-1} [u] = 0$
with $j\in\{0,\ldots,N-1\}$, then $\bfM^j[u]$ is $N$-harmonic in $\ball$.
\end{corollary}

\begin{proof}
Since $\bfL_{N-j-1} [u] = 0$, using the operator identity \eqref{eqn:commu2}, we have
\begin{align*}
\bfL_{N-1} [\bfM^j[u]] ~=~& \bfM^j [\bfL_{N-j-1}[u]] + 4j(j-2N+1) \bfM^{j-1}[u]\\
=~& 4j(j-2N+1) \bfM^{j-1}[u].
\end{align*}
We proceed iteratively and discover that for $k=1,\ldots,N$,
\begin{equation}\label{eqn:iterated}
\bfL_{N-k}\cdots \bfL_{N-1} [\bfM^j [u]] ~=~ 4^k (j-k+1)_k (j-2N+1)_k \bfM^{j-k} [u],
\end{equation}
where $(a)_0:=1$ and $(a)_k:=a(a+1) \cdots(a+k-1)$ for $k=1,2,\ldots$
are the ascending Pochhammer symbols. When $k>j$, the right hand side of \eqref{eqn:iterated} vanishes.
In particular, when $k=N$, \eqref{eqn:iterated} reads
\[
\bfL_{0}\cdots \bfL_{N-1} [\bfM^j [u]] ~=~ 0.
\]
In view of \eqref{eqn:factorization}, this implies that $\bfM^j[u]$ is $N$-harmonic in $\ball$.
\end{proof}

\begin{corollary}\label{cor:L_N-1action}
If $u$ is $N$-harmonic in $\ball$, then $\bfL_{N-1} [u]$ is
$(N-1)$-harmonic. If $N=1$, this should be interpreted as $\bfL_0 [u] = 0$.
\end{corollary}

\subsection{Special solutions of the equation $\bfL_{\theta}[u]=0$}

For $\zeta\in \sphere$, let
\begin{equation}\label{eqn:poisson}
P_{\theta}(x, \zeta) := C_{\theta} \frac {(1-|x|^2)^{1+2\theta}}
{|x-\zeta|^{n+2\theta}}, \qquad x\in \ball,
\end{equation}
where
\[
C_{\theta} := \frac {\Gamma(n/2+\theta) \Gamma(1+\theta)}
{\Gamma(n/2) \Gamma(1+2\theta)}.
\]

\begin{lemma}\label{lem:LPoisson}
Let $\theta \in \bbR$. Then
\begin{equation}
\bfL_{\theta} [P_{\theta}(\cdot, \zeta)]=0
\end{equation}
holds for any fixed $\zeta\in \sphere$.
\end{lemma}

\begin{proof}
In view of \eqref{eqn:commu5}, it suffices to show that
\[
\bfL_{-1-\theta} \left[\frac {1}{|x - \zeta|^{n+2\theta}}\right]=0,
\]
where the differentiation is with respect to $x$. Simple calculations yield
\begin{align}
\Delta \left[\frac {1}{|x - \zeta|^{n+2\theta}}\right] ~=~& (2+2\theta)(n+2\theta)\ \frac {1} {|x - \zeta|^{n+2\theta+2}}\\
\intertext{and}
\bfR\left[\frac {1}{|x - \zeta|^{n+2\theta}}\right] ~=~&  (-n-2\theta)\ \frac {|x|^2 - x\cdot\zeta} {|x - \zeta|^{n+2\theta+2}}.
\end{align}
It follows that
\begin{align*}
\bfL_{-1-\theta} \left[\frac {1}{|x - \zeta|^{n+2\theta}}\right]
~=~& (2+2\theta)(n+2\theta)\ \frac {1-|x|^2} {|x - \zeta|^{n+2\theta+2}} \\
& \quad + 4 (-1-\theta) (-n - 2\theta)\ \frac {|x|^2 - x\cdot\zeta} {|x - \zeta|^{n+2\theta+2}}\\
& \quad + 2(-1-\theta) [n-2-2(-1-\theta)]\ \frac {1}{|x-\zeta|^{n+2\theta}}\\
=~& 0,
\end{align*}
as desired.
\end{proof}

For every function $f\in L^1(\sphere, d\sigma)$ we define a function $\bfP_{\theta}[f]$ on $\ball$ as
follows.
\[
\bfP_{\theta} [f] (x) ~:=~ \int\limits_{\sphere} P_{\theta}(x,\zeta) f(\zeta) d\sigma(\zeta),
\qquad x\in \ball.
\]
The function $\bfP_{\theta} [f]$ will be called the $\theta$-Poisson integral of $f$.

\begin{lemma}[{\cite[Theorem 2.4]{LP04}}]
Let $\theta > -1/2$. The Dirichlet problem
\[
\begin{cases}
\bfL_{\theta} [u] =0, & \text{in }\; \ball\\
u=f,& \text{on }\; \bbS
\end{cases}
\]
has a unique solution, which is given by $u=P_{\theta}[f]$.
\end{lemma}

We consider the hypergeometric differential equation
\begin{equation}\label{eqn:hypergeqn}
z(1- z) f^{\prime\prime} (z) + \left[c - \left( a+b+1\right) z\right]
f^{\prime} (z) - ab f(z) = 0,
\end{equation}
where $a,b,c$ are complex parameters.
For $c\neq 0,-1,-2,\ldots$, the hypergeometric function is defined by
the power series
\begin{equation}\label{eqn:hypergeo}
\hyperg{a}{b}{c}{z} ~:=~ \sum_{k=0}^{\infty} \frac {(a)_k(b)_k}{(c)_k} \frac {z^k}{k!}, \qquad |z|<1,
\end{equation}
where $(a)_0:=1$ and $(a)_k:=a(a+1) \cdots(a+k-1)$ for $k=1,2,\ldots$.
It is well-known and straightforward to check that the function
$\hyperg{a}{b}{c}{z}$ satisfies the equation \eqref{eqn:hypergeqn} in the unit disk
$|z|<1$. %and the function
%\[
%z~\longmapsto~ (1-z)^{c-a-b} \hyperg {c-a}{c-b}{c+1-a-b}{1-z}
%\]
%solves the equation \eqref{eqn:hypergeqn} in the punctured unit disk $0<|z|<1$.
See \cite{AAR99} for a complete account on the subject.

\begin{lemma}\label{lem:Phi_theta}
The function
\begin{equation}\label{eqn:Phi_theta}
\Phi_{\theta} (x) ~:=~ \hyperg {-\theta} {\frac {n}{2}-1-\theta} {\frac {n}{2}} {|x|^2}
\end{equation}
solves the equation $\bfL_{\theta}[u] =0$ in $\ball$.%;
\end{lemma}

\begin{proof}
In the spherical-polar coordinates $x=r\zeta$, $r>0$, $\zeta\in \sphere$,
the Laplace operator $\Delta$ can be written as
\begin{equation}
\Delta ~=~ \frac {\partial^2}{\partial r^2} + \frac {n-1}{r} \frac {\partial}{\partial r}
+ \frac {1}{r^2} \Delta_{\sphere},
\end{equation}
where
\[
\Delta_{\sphere} ~:=~ \sum_{i=1}^{n-1} \frac {\partial^2}{\partial \zeta_i^2}
- \sum_{i=1}^{n-1} \sum_{j=1}^{n-1} \zeta_i \zeta_j \frac {\partial^2}{\partial \zeta_i \partial \zeta_j}
- (n-1) \sum_{i=1}^{n-1} \zeta_i \frac {\partial}{\partial \zeta_i}
\]
is the Laplace-Beltrami operator on the unit sphere $\sphere$. See for instance \cite[Lemma 1.4.1]{DX13}.

Now we consider radial solutions of the equation $\bfL_{\theta}[u]=0$.
Suppose that $u(x)=f(|x|^2)$, where $f$ is a $C^2$ function on the interval $(0,1)$.
Then, with $r = |x|$,
\begin{align*}
\Delta u(x) ~=~& 4 r^2 f^{\prime\prime} (r^2) + 2n f^{\prime}(r^2),\\
\bfR[u](x) ~=~& 2r^2 f^{\prime}(r^2)
\end{align*}
and hence
\begin{align*}
\bfL_{\theta}[u](x) ~=~ 4r^2 (1-r^2) f^{\prime\prime} (r^2) + \left[2n (1-r^2) + 8\theta r^2\right] f^{\prime}(r^2)
+ 2\theta (n-2-2\theta) f(r^2).
\end{align*}
Therefore, the differential equation $\bfL_{\theta}[u] =0$ deduces to% that $f$ satisfies the differential equation
\begin{equation}\label{eqn:hypergeqn2}
z(1- z) f^{\prime\prime} (z) + \left\{ \frac {n}{2} - \left( \frac {n}{2} - 2\theta\right) z\right\}
f^{\prime} (z) + \theta \left( \frac {n}{2} - 1 -\theta\right) f(z) = 0.
\end{equation}
This is the hypergeometric differential equation, with parameters
\[
a=-\theta,\quad b=\frac {n}{2} - 1 -\theta, \quad c=\frac {n}{2}.
\]
The hypergeometric function
\[
\hyperg {-\theta} {\frac {n}{2} - 1 -\theta} {\frac {n}{2}} {z}
\]
satisfies the equation \eqref{eqn:hypergeqn2} in the unit disk
$|z|<1$, and hence the function $\Phi_{\theta}$ solves the equation $\bfL_{\theta}[u] =0$ in $\ball$.
\end{proof}

The following lemma is immediate from the Gauss summation theorem for the hypergeometric function.
We include a proof here for the reader's convenience.

\begin{lemma}\label{lem:Phithetabounded}
$\Phi_{\theta}$ is bounded on $\ball$ if and only if $\theta>-1/2$.
\end{lemma}

\begin{proof}
By definition,
\[
\Phi_{\theta}(x) ~=~
\sum_{k=0}^{\infty}\frac{(-\theta)_k (n/2 - 1 -\theta)_k} {(1)_k (n/2)_k }
\ |x|^{2k}.
\]
It is easy to see that the coefficients in the series are
of order $k^{-2\theta-2}$ as $k\to \infty$, and the assertion of the lemma follows.
\end{proof}

\begin{corollary}\label{cor:PhiinPH}
Suppose that $0<p<\infty$, $N\in \bbN$ and $j\in \{0,1,\ldots,N-1\}$.
The function $\bfM^{j}[\Phi_{N-j-1}]$ is in $\PH_{N, \alpha}^p (\ball)$ for any $\alpha>-1-jp$.
\end{corollary}

\begin{proof}
By Lemma \ref{lem:Phi_theta} and Corollary \ref{cor:LthetaandNharmonic}, the function
$\bfM^{j}[\Phi_{N-j-1}]$ is $N$-harmonic. In view of Lemma \ref{lem:Phithetabounded},
the function $\Phi_{N-j-1}$ is bounded in $\ball$, so it is easy to check that
$\bfM^{j}[\Phi_{N-j-1}]$ is in $\PH_{N, \alpha}^p (\ball)$ for any $\alpha>-1-jp$.
\end{proof}

\subsection{Mapping properties of $\bfL_{\theta}$}

We will next analyse the image of $\PH^p_{N,\alpha}(\ball)$ under $\bfL_{\theta}$.

\begin{lemma}[{\cite[Lemma 5]{Pav97}}]\label{lem:Pav97-0}
Suppose that $0<p<\infty$ and $u$ is $N$-harmonic in $\ball$. Then
\begin{equation}
|u(x)|^p ~\lesssim~  r^{-n} \int\limits_{B(x,r)} |u(y)|^p dV(y)
\end{equation}
for all $x\in \ball$ and $r\in (0,1)$, where the implicit constant depends only on $p$, $N$ and $n$.
\end{lemma}

\begin{lemma}\label{lem:ptwsest}
Suppose $0<p<\infty$ and $\alpha\in \bbR$. Then
\begin{equation}
|u(x)| ~\lesssim~  (1-|x|^2)^{-(n+\alpha)/p} \|u\|_{p,\alpha}
\end{equation}
for all $u\in \PH_{N,\alpha}^p (\ball)$ and $x\in \ball$.
\end{lemma}

\begin{proof}
Let $u\in \PH_{N,\alpha}^p (\ball)$ and $x\in \ball$ be fixed.
Applying Lemma \ref{lem:Pav97-0} with $r=\frac {1}{2} (1-|x|)$, we have
\begin{equation}\label{eqn:ptwsest1}
|u(x)|^p ~\lesssim~ (1-|x|)^{-n} \int\limits_{B(x,\frac {1}{2} (1-|x|))} |u(y)|^p dV(y)
\end{equation}
Note that if $y\in B(x,\frac {1}{2} (1-|x|))$ then $1-|y|^2 \approx 1-|x|^2$. It follows from
\eqref{eqn:ptwsest1} that
\begin{align*}
|u(x)|^p ~\lesssim~& (1-|x|)^{-n-\alpha} \int\limits_{B(x,\frac {1}{2} (1-|x|))} |u(y)|^p (1-|y|^2)^{\alpha} dV(y)\\
\lesssim~&  (1-|x|)^{-n-\alpha} \|u\|_{p,\alpha}^p
\end{align*}
as desired.
\end{proof}

\begin{lemma}[{\cite[Lemma 6]{Pav97}}]\label{lem:Pav97}
Suppose that $0< p <+\infty $, $N\in \bbN$ and $\alpha \in \bbR$. If $u\in \PH_{N,\alpha}^p(\ball)$ then
$\partial_j u\in \PH_{N,\alpha+p}^p(\ball)$, $j=1,\ldots, n$. %Moreover, there exists a constant
\end{lemma}

\begin{corollary}\label{cor:Pav97-2}
Suppose that $0< p <+\infty $, $N\in \bbN$  and $\alpha \in \bbR$. If $u\in \PH_{N,\alpha}^p(\ball)$ then
$\Delta^k u\in \PH_{N-k,\alpha+2kp}^p(\ball)$ for each $k\in \{1,\ldots, N-1\}$.
\end{corollary}

\begin{proposition}\label{prop:mappingprpty}
Suppose that $0< p <+\infty $, $N\in \bbN$  and $\alpha \in \bbR$. If $u\in \PH^p_{N,\alpha}(\ball)$
then $\bfL_{\theta} [u] \in \PH^p_{N,\alpha+p}(\ball)$.
\end{proposition}

\begin{proof}
Suppose $u\in \PH^p_{N,\alpha}(\ball)$. We show that each term on the right hand side of \eqref{eqn:Ltheta} belongs to $\PH^p_{N,\alpha+p}(\ball)$. First, by Corollary \ref{cor:Pav97-2}, we have $\Delta u \in \PH^p_{N-1,\alpha+2p}(\ball)$ and hence
$\bfM \Delta u \in \PH^p_{N,\alpha+p}(\ball)$.
Next, it is easy to check that $\Delta^N \bfR = \bfR \Delta^N + 2N \Delta^N$. So $\bfR[u]$ is $N$-harmonic.
It then follows from Lemma \ref{lem:Pav97} that $\bfR [u] \in \PH^p_{N,\alpha+p}(\ball)$.
We also have $u\in \PH^p_{N,\alpha+p}(\ball)$ because trivially $\PH^p_{N,\alpha}(\ball) \subset \PH^p_{N,\alpha+p}(\ball)$.
By linearity, we are done.
\end{proof}

\section{Proof of Theorem \ref{thm:mdfalmansi}}

%Having proved the identities \eqref{eqn:commu2}, \eqref{eqn:commu4} and Proposition \ref{prop:mappingprpty},
%the proof of Theorem \ref{thm:main1} follows the same line of reasoning as that of Theorem 3.4 of \cite{BH14}.
%For the readers' convenience, we repeat it here.

\subsection*{Uniqueness}

It suffices to show that if
\begin{equation}\label{eqn:uniq0}
\sum_{j=0}^{N-1} \bfM^j[w_j] = 0
\end{equation}
with $w_j$ satisfying $\bfL_{N-j-1}[w_j]=0$, $j=0,\ldots, N-1$, then
all the functions $w_j$ vanish.

To prove this we proceed by induction on $N$. Clearly, when $N = 1$,
then \eqref{eqn:uniq0} just states that $w_0 = 0$, as needed.
For the induction step, assume the above assertion holds for
$N = N_0$.

Suppose now that
\begin{equation}\label{eqn:uniq1}
\sum_{j=0}^{N_0} \bfM^j[w_j] = 0
\end{equation}
with $w_j$ satisfying $\bfL_{N_0-j}[w_j]=0$, $j=0,\ldots, N_0$.
Applying $\bfL_{N_0}$ to both sides of \eqref{eqn:uniq1} and using the operator identity
\eqref{eqn:commu2}, we obtain
\[
\sum_{j=0}^{N_0} \big\{\bfM^j\bfL_{N_0-j}[w_j] + 4j(j-2N_0-1) \bfM^{j-1}[w_j]\big\} ~=~ 0.
\]
Since $\bfL_{N_0-j}[w_j]=0$, $j=0,\ldots, N_0$, after setting $\widetilde{w}_j := (j+1)(j-2N_0) w_{j+1}$, the equation becomes
\[
\sum_{j=0}^{N_0-1} \bfM^j[\widetilde{w}_j] = 0.
\]
By the induction hypothesis, we have that $\widetilde{w}_j =0$ for all $j=0,\ldots, N_0-1$.
As a consequence, $w_j =0$ for all $j=1,\ldots, N_0$. In view of \eqref{eqn:uniq1}, this in turn implies
$w_0=0$. The uniqueness part of the theorem is proved.

\subsection*{Existence}

Again, we argue by induction on $N$.
The case $N = 1$ is trivial. For the induction step, assume the assertion of the theorem holds for
$N = N_0>1$.

Now, we suppose that $u$ is a $(N_0+1)$-harmonic function on $\ball$. Then
$\bfL_{N_0}[u]$ is $N_0$-harmonic on $\ball$, by Corollary \ref{cor:L_N-1action}.
Thus, by the induction hypothesis,
\[
\bfL_{N_0}[u] ~=~ \sum_{j=0}^{N_0-1} \bfM^j [v_j],
\]
with $v_{j}$ satisfying $\bfL_{N_0-j-1}[v_j]=0$ for $j=0,\ldots, N_0-1$.
%Moreover,
%\begin{equation}\label{eqn:integrabilityofv_j}
%v_j \in \PH_{N_0-j,\alpha + (j+1)p}^p (\ball).
%\end{equation}
Putting
\[
V ~:=~ \frac {1}{4}  \sum_{j=0}^{N_0-1} \frac {1}{(j+1)(2N_0 -j)} \bfM^{j+1} [v_j],
\]
we have
\begin{align*}
\bfL_{N_0} [u+V] ~=~& \sum_{j=0}^{N_0-1} \left\{ \bfM^j [v_j]+ \frac {1}{4(j+1)(2N_0 -j)} \bfL_{N_0}\bfM^{j+1} [v_j]\right\}\\
=~& \sum_{j=0}^{N_0-1} \bigg\{ \bfM^j [v_j]+ \frac {1}{4(j+1)(2N_0 -j)} \Big(\bfM^{j+1}\bfL_{N_0-j-1} [v_j]\\
& \qquad -4(j+1)(2N_0-j)\bfM^j [v_j] \Big) \bigg\}\\
=~& 0,
\end{align*}
where we used the operator identity \eqref{eqn:commu2} and that $\bfL_{N_0-j-1}[v_j]=0$ for $j=0,\ldots, N_0-1$.
We now define
\begin{align*}
w_0 ~:=~&u+V,\\
w_j ~:=~& - \frac {1}{4j(2N_0 -j+1)}\, v_{j-1}, \quad j=1,\ldots, N_0.
\end{align*}
Then $w_j$ satisfies $\bfL_{N_0-j}[w_j]=0$ for $j=0,1,\ldots, N_0$, and hence
\[
u ~=~ w_0 - V ~=~ \sum_{j=0}^{N_0} \bfM^j [w_j]
\]
is the modified Almansi representation of $u$. This completes the proof.

\section{Proof of Theorem \ref{thm:cellular}}

%\begin{corollary}
%Let $0<p<\infty$, $N\in \bbN$ and $\alpha\in \bbR$. Then every
%$u\in\mathrm{PH}_{N,\alpha}^{p}(\ball)$ has a unique decomposition
%\[
%u=w_{0}+ \bfM [w_{1}]+ \cdots +\bfM^{N-1}[w_{N-1}],
%\]
%where each term $\bfM^{j}[w_{j}]$ is in $\mathrm{PH}_{N,\alpha}^{p}(\ball)$, while the functions $w_{j}$ are $(N-j)$-harmonic and solve  $\bfL_{N-j-1}[w_{j}]=0$ on $\ball$, for $j=0,\ldots,N-1$.
%\end{corollary}
With Theorem \ref{thm:mdfalmansi} at hand, it remains to show that each term $\bfM^j[w_j]$ in
\eqref{eqn:celluardecomp} is in the space $\mathrm{PH}_{N,\alpha}^{p}(\ball)$.

We first show that $\bfM^{N-1}[w_{N-1}]  \in \PH_{N,\alpha}^p (\ball)$.
%First, each term $\bfM^j[w_j]$ in \eqref{eqn:celluardecomp} is $N$-harmonic, by Corollary \ref{cor:LthetaandNharmonic}.
%
%The proof is essentially the same as in
%We shall not repeat the argument here.
%To avoid repetition,
%
%We proceed to show how to obtain each term of the expansion from the
%given function $u$:
For any fixed $k\in \{1,\ldots, N-1\}$, it follows from
\eqref{eqn:celluardecomp} and \eqref{eqn:iterated} that
\begin{align}\label{eqn:recover_w_j}
\bfL_{N-k}\cdots \bfL_{N-1} [u] ~=~& \sum_{j=0}^{N-1} \bfL_{N-k}\cdots \bfL_{N-1} [\bfM^j [w_j]] \\
~=~& \sum_{j=0}^{N-1} 4^k (j-k+1)_k (j-2N+1)_k \bfM^{j-k} [w_j] \notag\\
~=~& \sum_{j=k}^{N-1} 4^k (j-k+1)_k (j-2N+1)_k \bfM^{j-k} [w_j]. \notag
\end{align}
In particular, taking $k=N-1$, this leads to
\[
\bfL_{1}\cdots \bfL_{N-1} [u] ~=~ (-4)^{N-1} (N-1)! N! \, w_{N-1}.
\]
Thus, repeated application of Proposition \ref{prop:mappingprpty} yields that $w_{N-1} \in \PH_{N,\alpha + (N-1)p}^p (\ball)$.
But $w_{N-1}$ satisfies $\bfL_{0}[w_{N-1}]=0$, i.e., $w_{N-1}$ is harmonic on $\ball$.
Hence in fact $w_{N-1}  \in \PH_{1,\alpha + (N-1)p}^p (\ball)$, which in turn implies $\bfM^{N-1}[w_{N-1}]  \in \PH_{N,\alpha}^p (\ball)$,
in view of Corollary \ref{cor:LthetaandNharmonic}.

Now we put
\[
v ~:=~ u-\bfM^{N-1}[w_{N-1}].
\]
Note that $v\in \PH_{N,\alpha}^p (\ball)$ and
\[
v = \sum_{j=0}^{N-2} \bfM^{j}[w_j].
\]
For fixed $k\in \{1,\ldots,N-2\}$, we proceed in the same way as for \eqref{eqn:recover_w_j} to obtain
\[
\bfL_{N-k}\cdots \bfL_{N-1} [v] ~=~ \sum_{j=k}^{N-2} 4^k (j-k+1)_k (j-2N+1)_k \bfM^{j-k} [w_j].
\]
When $k=N-2$, this reads
\[
\bfL_{2}\cdots \bfL_{N-1} [v] ~=~ (-4)^{N-2} (N-2)! \frac {(N+1)!}{3!} \,w_{N-2}.
\]
Then, by arguments similar to those for $w_{N-1}$, we obtain $\bfM^{N-2}[w_{N-2}]\in \PH_{N,\alpha}^p (\ball)$.

%After repeating this procedure $N-2$ times, we get

Continuing inductively in this manner, we find that $\bfM^{j}[w_{j}]\in \PH_{N,\alpha}^p (\ball)$ for $j=N-3,
\ldots,0$, and the proof is complete.

%%%%%%%%%%%%%%%%%%%%%%%%%%%%%

\section{Proof of Theorem \ref{thm:main2}: Part 1}

%For convenience, we reformulate Theorem \ref{thm:main2} as follows.
When $n\geq 3$, the formula \eqref{eqn:crit-type2} follows immediately from \eqref{eqn:crit-type}.
So, we only prove \eqref{eqn:crit-type}. For convenience, we divide the proof into two separate theorems.

\begin{theorem}\label{thm:main2refml}
Suppose that $0 < p < \infty$, $N\in \bbN$ and $\alpha$ is real. Then
\[
\PH_{N,\alpha}^{p}(\ball)=\{0\} \quad \Longrightarrow \quad \alpha \leq \min_{j:0\leq j \leq N} b_{j,N}(p).
\]
\end{theorem}

\begin{theorem}\label{thm:main2refm2}
Suppose that $N\in \bbN$, $\frac{n-2}{n-1}\leq  p < \infty$ and $\alpha$ is real. Then
\[
\alpha \leq \min_{j:0\leq j \leq N} b_{j,N}(p) \quad \Longrightarrow \quad \PH_{N,\alpha}^{p}(\ball)=\{0\}.
\]
\end{theorem}

Note that even for $n\geq 3$ we do not require that $p \geq \frac{n-2}{n-1}$ in Theorem \ref{thm:main2refml}.

This section is devoted to the proof of Theorem \ref{thm:main2refml}. Theorem \ref{thm:main2refm2} will be proved in Section 6.

Given $N\in \bbN$ and $j\in \{1,\ldots, N\}$, let
\begin{equation}\label{eqn:UjN}
U_{j,N}(x) ~:=~ \frac {(1-|x|^{2})^{N+j-1}} {|x-e_{1}|^{n+2(j-1)}}, \qquad x\in \ball,
\end{equation}
where $e_1=(1,0,\cdots,0)$ is the first coordinate vector in $\bbR^n$, while for $j=0$ we put
\[
U_{0,N}(x) ~:=~ (1-|x|^2)^{N-1}.
\]

%%%%%%%%%%%%%%%%%%%%%%%%%%%%

\begin{lemma}
For $N = 1, 2, 3, \ldots$ and $j = 0, 1, \ldots, N$, the functions $U_{j,N}$ are all
$N$-harmonic in $\ball$.
\end{lemma}

\begin{proof}
The function $U_{0,N}$ is clearly $N$-harmonic in $\ball$. For $j\in \{1,\ldots, N\}$,
note that $U_{j,N}=\bfM^{N-j}[P_{j-1}(\cdot, e_1)]$, where $P_{\theta}$ is defined as in \eqref{eqn:poisson}.
By Lemma \ref{lem:LPoisson}, $P_{j-1}(\cdot, e_1)$ solves $\bfL_{j-1}[u]=0$. Hence, by Corollary \ref{cor:LthetaandNharmonic},
$U_{j,N}$ is $N$-harmonic in $\ball$.
\end{proof}

%%%%%%%%%%%%%%%%%%%%%%%%%

\begin{lemma}\label{lem:finiteintgl}
Let $a,b \in \bbR$. The integral
\[
I(a,b) := \int\limits_{\ball} \frac {(1-|x|^2)^a} {|x-e_{1}|^{n+a+b}} dV(x)
\]
is finite if and only if $a>-1$ and $b<0$. Moreover, if $a > -1$ and $b < 0$ then
\[
I(a,b) ~=~ \frac {\pi^{n/2}\Gamma(1+a)\Gamma(-b)}{\Gamma\big((n+a-b)/2\big)\Gamma \big((2+a-b)/2\big)}.
\]
\end{lemma}

\begin{proof}
We first recall the following formula (see \cite[lemma 2.1]{LP04}):
\[
\int\limits_{\sphere} \frac {d\sigma(\zeta)}{|y-\zeta|^{2 t}} ~=~ \hyperg {t}{t-\frac {n}{2}+1}{\frac {n}{2}} {|y|^2},
\qquad y\in \ball,
\]
where $t$ is a real parameter.
By integrating in polar coordinates and using the above formula, we find that
\begin{align*}
I(a,b) ~=~& \omega_{n-1} \int\limits_{0}^{1} r^{n-1} (1-r^2)^a
\left\{ \int\limits_{\sphere} \frac {d\sigma(\zeta)}{|r e_1 -\zeta|^{n+a+b}}\right\} dr\\
~=~& \omega_{n-1} \int\limits_{0}^{1} r^{n-1} (1-r^2)^a
\hyperg {\frac {n+a+b}{2}} {\frac {2+a+b}{2}} {\frac {n}{2}} {r^2} dr\\
=~& \frac {\omega_{n-1}}{2} \sum_{j=0}^{\infty} \frac {\big((n+a+b)/2\big)_j
\big((2+a+b)/2\big)_j} {\big(n/2\big)_j (1)_j} \int\limits_0^1 r^{j+n/2-1}(1-r)^a dr,
\end{align*}
where $\omega_{n-1}:= 2\pi^{n/2}/\Gamma(n/2)$ stands for the area of the unit sphere $\sphere$.
For $a\leq -1$, we have $I(a, b) = +\infty$. For $a>-1$, we evaluate the (Beta) integral to obtain
\begin{equation}\label{eqn:Iab}
I(a,b) ~=~ \frac {\pi^{n/2}\Gamma(1+a)}{\Gamma(n/2+1+a)} \sum_{j=0}^{\infty} \frac {\big((n+a+b)/2\big)_j
\big((2+a+b)/2\big)_j} {\big(n/2+1+a\big)_j (1)_j}.
\end{equation}
Using the well-known Stirling formula
\[
\frac {\Gamma(j+t)} {\Gamma(j+s)} ~\sim~ j^{t-s} \quad \text{as}\quad  j \to +\infty,
\]
we find that the sum on the right-hand side of \eqref{eqn:Iab} converges if and only if
\[
\sum\limits_{j=1}^{\infty} j^{b-1} < \infty,
\]
if and only if $b<0$.

Now we assume that $a>-1$ and $b<0$. Then the sum on the right-hand side of \eqref{eqn:Iab} equals
\[
\hyperg {\frac {n+a+b}{2}} {\frac {2+a+b}{2}} {\frac {n}{2} +1+a} {1}
~=~ \frac {\Gamma(n/2+1+a) \Gamma(-b)}{\Gamma\big((n+a-b)/2\big)\Gamma \big((2+a-b)/2\big)},
\]
where we have used the well-known formula of Gauss
\begin{align*}
&\hyperg{\alpha}{\beta}{\gamma}{1} ~=~ \frac{\Gamma(\gamma)\Gamma(\gamma-\alpha-\beta)}
{\Gamma(\gamma-\alpha)\Gamma(\gamma-\beta)}, \qquad \RePt(\gamma-\alpha-\beta)>0.
\end{align*}
This completes the proof.
\end{proof}

\begin{lemma}\label{lem:testtriviality} %\marginpar{Corresponds to Lem 5.2 in BH}
For each fixed $N\in \bbN$ and $j\in \{0,1,\ldots, N\}$, the function $U_{j,N}$ is in $\PH_{N,\alpha}^{p}(\ball)$ if and only if $\alpha >b_{j,N}(p)$.
\end{lemma}

\begin{proof}
Clearly, $U_{0,N}\in \PH_{N,\alpha}^{p}(\ball)$ if and only if $(N-1)p+\alpha > -1$, which is exactly $\alpha > b_{0,N}(p)$.
For $j\in \{1,\ldots,N\}$, to decide when $U_{j,N}\in \PH_{N,\alpha}^{p}(\ball)$, we note that
\[
\|U_{j,N}\|_{p,\alpha}^{p} ~=~ \int\limits_{\ball} \frac {(1-|x|^2)^{(N+j-1)p+\alpha}} {|x-e_{1}|^{(n+2j-2)p}} dV(x),
\]
which is finite if and only if
\begin{equation}\label{eqn:condn1}
\begin{cases} (N+j-1)p+\alpha > -1,\\
(n+2j-2)p-n-(N+j-1)p-\alpha < 0,
\end{cases}
\end{equation}
in view of Lemma \ref{lem:finiteintgl}. The claim follows, since the condition \eqref{eqn:condn1} is exactly the same as $\alpha >b_{j,N}(p)$.
\end{proof}

%%%%%%%%%%%%%%%%%%%%%%

Lemma \ref{lem:testtriviality} shows that if $\alpha$ satisfies
\[
\alpha > \min_{j:0\leq j\leq N} b_{j,N}(p),
\]
then
one of the functions $U_{0,N}, U_{1,N}, \ldots, U_{N,N}$ will be in $\mathrm{PH}_{N,\alpha}^p(\ball)$,
so that in particular, $\mathrm{PH}_{N,\alpha}^p(\ball) \neq \{0\}$.
This completes the proof of Theorem \ref{thm:main2refml}.

\section{Preliminaries for the proof of Theorem \ref{thm:main2refm2}}

%According to the classical Almansi representation, $u$ is
%$N$-harmonic if and only if it is of the form
%\begin{equation*}
%u(x) = u_{0}(x) + |x|^2 u_{1}(x) + \cdots + |x|^{2N-2} u_{N-1}(x),
%\end{equation*}
%where all the functions $u_j$ are harmonic in $\ball$ (see, e.g., Section V of \cite{ACL83}).
%This can be rearranged to obtain
%\begin{equation}\label{eqn:almansi}
%u(x) = v_{0}(x) + (1-|x|^2)v_{1}(x) + \cdots + (1-|x|^2)^{N-1} v_{N-1}(x),
%\end{equation}
%where the functions $v_j$ are given as
%\[
%v_j := (-1)^{j} \sum_{k=j}^{N-1} \binom{k}{j} u_k,
%\]
%which are harmonic functions on $\ball$.

The following result, which generalizes Proposition 4.11 in \cite{BH14},
provides us with condition that guarantees that an $N$-harmonic function $u(x)$ can be written as $(1-|x|^2) \widetilde{u}(x)$, where $\widetilde{u}$ is
$(N-1)$-harmonic.

\begin{proposition}\label{prop:prop411inBH}
%\marginpar{Corresponds to Prop 4.11 in BH}
Suppose that $0< p <\infty$, $\alpha \leq \min\{(n-1)p-n,-1\}$ and $N\in \bbN$.
Suppose $u\in \PH_{N,\alpha}^{p}(\ball)$.
\begin{enumerate}
\item[(i)]
If $N=1$, then $u = 0$;
\item[(ii)]
if $N\geq 2$ then $u$ has the form $u=\bfM[\widetilde{u}]$
with $\widetilde{u}\in \PH_{N-1,\alpha+p}^{p}(\ball)$.
\end{enumerate}
\end{proposition}

\begin{proof}
We first show that
\begin{equation}\label{eqn:liminf}
\liminf_{r \to 1^{-}} \int\limits_{\sphere} |u(r\zeta)| d\sigma(\zeta) ~=~ 0.
\end{equation}

\subsubsection*{Case 1: $0<p<1$}
Since then $\alpha \leq (n-1)p-n$, we have
$u\in  \PH_{N,\alpha}^{p}(\ball) \subset  \PH_{N,(n-1)p-n}^{p}(\ball)$, and
\begin{align*}
\|u\|_{p,(n-1)p-n}^p ~\leq~ \|u\|_{p,\alpha}^p ~<~ +\infty.
\end{align*}
By Lemma \ref{lem:ptwsest}, we have
\begin{equation}
\sup_{x\in \ball} |u(x)|^p (1-|x|^{2})^{(n-1)p} ~\lesssim~  \|u\|_{p,(n-1)p-n}^{p} < +\infty.
\end{equation}
Thus,
\begin{align*}
\|u\|_{1,-1} ~=~& \int\limits_{\ball} |u(x)|^{p}(1-|x|^2)^{(n-1)p-n}
\left\{ |u(x)|^p(1-|x|^2)^{(n-1)p} \right\}^{(1-p)/p} dV(x)\\
\leq~&  \|u\|_{p,(n-1)p-n}^{p} \left\{ \sup_{x\in \ball} |u(x)|^p (1-|x|^{2})^{(n-1)p}\right\}^{(1-p)/p} ~<~ +\infty.
\end{align*}
Now we prove \eqref{eqn:liminf} by contradiction. Assume that
\begin{equation}
\liminf_{r \to 1^{-}} \int\limits_{\sphere} |u(r\zeta)| d\sigma(\zeta) ~>~ 0.
\end{equation}
Then there exists a $\delta>0$ such that
\[
\inf_{1-\delta< r<1} \int\limits_{\sphere} |u(r\zeta)| d\sigma(\zeta) ~>~ 0.
\]
It follows that
\begin{align*}
\|u\|_{1,-1} ~=~& \int\limits_{0}^{1} \frac{r^{n-1}}{1-r^2} \left\{\int\limits_{\bbS}| u(r\zeta)| \ d\sigma(\zeta)\right\} dr \notag\\
\geq~ & \left\{\int\limits_{1-\delta}^{1} \frac{r^{n-1} dr}{1-r^2} \right\}
\left\{\inf_{1-\delta< r<1} \int\limits_{\sphere} |u(r\zeta)| d\sigma(\zeta)\right\} = +\infty.
\end{align*}
A contradiction.

\subsubsection*{Case 2: $1\leq p<+\infty$}
Since $\alpha\leq -1$, we have
\[
\|u\|_{p,-1}^p ~=~ \int\limits_{\ball} |u(x)|^p (1-|x|^2)^{-1} dV(x)
~\leq~ \int\limits_{\ball} |u(x)|^p (1-|x|^2)^{\alpha} dV(x) ~<~ +\infty.
\]
By the same elementary argument as above, we deduce that
\begin{equation}
\liminf_{r \to 1^{-}} \int\limits_{\sphere} |u(r\zeta)|^p d\sigma(\zeta) ~=~ 0.
\end{equation}
and \eqref{eqn:liminf} follows from this and an application of H\"older's inequality.

Now we proceed to prove the proposition.
By the alternative Almansi representation \eqref{eqn:yaalmansi}, we see that
\begin{equation*}
u(x) = v_{0}(x) + (1-|x|^2)v_{1}(x) + \cdots + (1-|x|^2)^{N-1} v_{N-1}(x),
\end{equation*}
where $v_0,v_1,\ldots, v_{N-1}$ are harmonic functions on $\ball$.
It follows that
\begin{align*}
\int\limits_{\sphere} u(r\zeta) \frac {1-|x|^2}{|x-\zeta|^n} d\sigma(\zeta) ~=~& \sum_{j=0}^{N-1} (1-r^2)^{j} \int\limits_{\sphere}
v_j(r\zeta) \frac {1-|x|^2}{|x-\zeta|^n} d\sigma(\zeta)\\
=~& \sum_{j=0}^{N-1} (1-r^2)^{j} v_j(rx).
\end{align*}
Letting $r\to 1^{-}$, we obtain
\begin{equation}
v_{0}(x) = \lim_{r\to 1^{-}} \int\limits_{\sphere} \frac{1-|x|^2}{|x-\zeta|^n} \, u(r\zeta)\ d\sigma(\zeta)
\end{equation}
for every $x\in \ball$. It follows that
\begin{eqnarray*}
|v_{0}(x)| &=& \lim_{r\to 1^{-}} \Bigg|\int\limits_{\sphere} \frac{1-|x|^2}{|x-\zeta|^n} \, u(r\zeta)\ d\sigma(\zeta)\Bigg|\\
&\leq& \frac{1+|x|}{(1-|x|)^{n-1}} \liminf_{r\to 1^{-}} \int\limits_{\sphere} |u(r\zeta)| d\sigma(\zeta) ~=~0
\end{eqnarray*}
for all $x\in \ball$. If $N=1$, we are done. If $N\geq 2$, we obtain instead that $u(x)=(1-|x|^2)\widetilde{u}(x)$ where
$$\widetilde{u}(x) ~:=~ \frac{u(x)}{1-|x|^2}=v_{1}(x)+(1-|x|^2)v_{2}(x)+\cdots+(1-|x|^2)^{N-2}v_{N-1}(x),$$
is $(N-1)$-harmonic. Moreover, this gives $\widetilde{u}\in \PH_{N-1,\alpha+p}^{p}(\ball)$.
\end{proof}

The following is a sufficient criterion for the triviality of a polyharmonic function. We note that the restriction $p\geq \frac{n-2}{n-1}$ enters the picture here.

\begin{proposition}\label{prop:prop412inBH}
%\marginpar{Corresponds to Prop 4.12 in BH}
Suppose that $\frac {n-2}{n-1} \leq  p< +\infty$ and $N\in \bbN$.
Then \( \PH_{N,\alpha}^{p}(\ball) = \{0\} \) for all $\alpha \leq -1-(2N-1)p$.
\end{proposition}

\begin{proof}
Since the spaces $\PH_{N,\alpha}^{p}(\ball)$ grow with $\alpha$,
we only need to prove the result when $\alpha=-1-(2N-1)p$.

The proof is by induction on $N$.
We first prove the claim for $N=1$:
\[
\mathrm{PH}_{1,-1-p}^{p}({\mathbb B}^n)=\{0\}.
\]
Let $u\in\mathrm{PH}_{1,-1-p}^{p}({\mathbb B}^n)$ be arbitrary. Then by Lemma \ref{lem:Pav97}, $\partial_j u\in \PH_{1,-1}^{p}(\ball)$,
$j=1,\ldots, n$.
Note that
\begin{equation}\label{eqn:normest}
\|\nabla u\|_{p,-1}^{p} ~=~ \omega_{n-1} \int\limits_{0}^{1} \frac{r^{n-1}}{1-r^2}
\bigg\{\int\limits_{\bbS}|\nabla u(r\zeta)|^{p}\ d\sigma(\zeta)\bigg\} dr,
\end{equation}
where, as usual,
\[
|\nabla u| := \bigg(\sum_{j=1}^n \left| \frac {\partial u}{\partial x_j}\right|^2 \bigg)^{1/2} \quad
\text{and} \quad \|\nabla u\|_{p,\alpha} := \big\| |\nabla u| \big\|_{p,\alpha}.
\]
Since $u$ is harmonic in $\ball$, $|\nabla u|^p$ is subharmonic when $p\geq \frac{n-2}{n-1}$, by \cite[Theorem A]{SW60}. Hence the function
\[
t ~\longmapsto~ \int\limits_{\bbS}|\nabla u(t\zeta)|^{p}\ d\sigma(\zeta)
\]
is increasing. It then follows from \eqref{eqn:normest} that
\[
\|\nabla u\|_{p,-1}^{p} ~\geq~ \bigg\{\int\limits_{t}^{1} \frac{r^{n-1}}{1-r^2} dr\bigg\}
\bigg\{\int\limits_{\bbS} |\nabla u(t\zeta)|^{p}\ d\sigma(\zeta)\bigg\}
\]
for every $0<t<1$.
Thus, $\|\nabla u\|_{p,-1}^{p}< + \infty$ forces
$\nabla u=0$, and hence $u$ must be constant.
As the only constant function in $\PH_{1,-1-p}^{p}(\ball)$ is the zero function, we obtain $u=0$.

For the induction step, we assume that the above assertion holds for
$N = N_0$:
\[
\mathrm{PH}_{N_0,-1-(2N_0-1)p}^{p}(\ball)=\{0\}.
\]
Let $u\in\mathrm{PH}_{N_0+1,-1-(2N_0+1)p}^{p}(\ball)$ be arbitrary. Put $v:= \Delta u$.
By Corollary \ref{cor:Pav97-2}, $v \in \mathrm{PH}_{N_0,-1-(2N_0-1)p}^{p}(\ball)$.
Then $v=0$, by the induction hypothesis. This means that $u$ is harmonic and
furthermore $u\in\mathrm{PH}_{1,-1-(2N_0+1)p}^{p}(\ball)$. But
$\mathrm{PH}_{1,-1-(2N_0+1)p}^{p}(\ball) \subset \mathrm{PH}_{1,-1-p}^{p}(\ball) =\{0\}$,
we find that $u=0$.
Consequently,
\[
\mathrm{PH}_{N_0+1,-1-(2N_0+1)p}^{p}(\ball) =\{0\}.
\]
The proof is complete.
\end{proof}

%%%%%%%%%%%%%%%%%%%%%%%%%%%%%%%%%%

\section{Proof of Theorem \ref{thm:main2}: Part 2}

In this section, we shall prove Theorem \ref{thm:main2refm2}, which together with Theorem
\ref{thm:main2refml} will complete the proof of Theorem \ref{thm:main2}.

For fixed $N\in \mathbb{N}$ and $j\in \{1,\cdots,N\}$, we define
\begin{equation}
a_{j,N}(p) ~:=~ \min\{b_{j,N}(p),-1-(N-j)p\},
\end{equation}
where
\begin{equation*}
b_{j,N}(p) ~:=~ \max\{-1-(N+j-1)p,-n-(N-j-n+1)p\},
\end{equation*}
as defined in \eqref{eqn:bjNp}.
Note that $a_{j,N}(p)=b_{j,N}(p)$ for $0<p<1$ and
\begin{equation}
\min_{j:1\leq j \leq N} a_{j,N}(p) = \min_{j:0\leq j \leq N} b_{j,N}(p).
\end{equation}

Thus, we can reformulate Theorem \ref{thm:main2refm2} as follows.

\begin{thmbis}{thm:main2refm2} \label{thm:main2refm3}
Suppose that $\frac{n-2}{n-1}\leq  p < \infty$ and $N\in \bbN$. Then
\[
\alpha \leq \min_{j:1\leq j \leq N} a_{j,N}(p) \quad \Longrightarrow \quad \PH_{N,\alpha}^{p}(\ball)=\{0\}.
\]
\end{thmbis}

According to Theorem \ref{thm:cellular}, any $u\in\mathrm{PH}_{N,\alpha}^{p}(\ball)$ can be uniquely written as
\begin{equation*}
u=w_{0}+ \bfM [w_{1}]+ \cdots +\bfM^{N-1}[w_{N-1}],
\end{equation*}
where each term $\bfM^{j}[w_{j}]$ remains in the space $\mathrm{PH}_{N,\alpha}^{p}(\ball)$, with $w_{j}$ solving
$\bfL_{N-j-1}[w_{j}]=0$ on $\ball$. Therefore, to show $u=0$, we just need to test each term
$\bfM^{j}[w_{j}]$ separately. Thus the proof of Theorem \ref{thm:main2refm3} reduces to proving the following proposition.

\begin{proposition}\label{lem:prop71inBH} %\marginpar{Corresponds to Prop 7.1 in BH}
Suppose that $\frac{n-2}{n-1} \leq p <+\infty$, $N\in \bbN$, $j\in \{0,1,\ldots,N-1\}$.
If $\alpha\leq a_{N-j,N}(p)$ and $u\in \PH_{N,\alpha}^{p}(\ball)$ is of the form $u=\bfM^{j}[w]$, with $w$ satisfying
$\bfL_{N-j-1}[w]=0$, then $u=0$.
\end{proposition}

\begin{proof}
It is clear that $w\in \PH_{N-j,\alpha+jp}^{p}(\ball)$. The assumption $\alpha\leq a_{N-j,N}(p)$
can be written as
\[
\alpha+jp ~\leq~ a_{N-j,N}(p)+jp ~=~ a_{N-j,N-j}(p).
\]
Let $N^{\prime}:=N-j$ and $\alpha^{\prime}:=\alpha+jp$. We are reduced to proving the following

\begin{claim}
Assume that $\alpha^{\prime}\leq a_{N^{\prime},N^{\prime}}(p)$.
If $w\in \PH_{N^{\prime},\alpha^{\prime}}^{p}(\ball)$
solves $\bfL_{N^{\prime}-1}[w]=0$, then $w=0$.
\end{claim}

First note that, in the case when $\frac{n-2}{n-1} \leq p <  \frac{n-1}{n+2N^{\prime}-2}$ (this is only possible if
$N^{\prime}=1$),
\[
a_{N^{\prime},N^{\prime}}(p) = -1-(2N^{\prime}-1)p.
\]
The assertion $w=0$ is then immediate from Proposition \ref{prop:prop412inBH}.

Now we assume that $p > \frac{n-1}{n+2N^{\prime}-2}$. Then
\[
a_{N^{\prime},N^{\prime}}(p)=\min\{(n-1)p-n,-1\}.
\]
Since $\alpha^{\prime}\leq a_{N^{\prime},N^{\prime}}(p)$, by Proposition
\ref{prop:prop411inBH}, $w$ can be written as $w=\bfM[\widetilde{w}]$, with $\widetilde{w}\in\PH_{N^{\prime}-1,\alpha^{\prime}+p}^{p}(\ball)$.
If $N^{\prime}=1$, this should be understood as $\widetilde{w}=0$ and we are done. If $N^{\prime}\geq 2$,
by Theorem \ref{thm:cellular}, $\widetilde{w}$
has a unique decomposition
\[
\widetilde{w} ~=~ \sum_{j=0}^{N^{\prime}-2} \bfM^{j}[v_{j}],
\]
where each term $\bfM^{j}[v_{j}] \in \PH_{N^{\prime}-1, \alpha^{\prime}+p}^{p}(\ball)$ with $v_j$ satisfying
$\bfL_{N^{\prime}-j-2}[v_{j}]=0$. This means that $w=\bfM[\widetilde{w}]$ has the expansion
\begin{equation}\label{eqn:eq7.1inBH}
w ~=~  \sum_{j=1}^{N^{\prime}-1} \bfM^{j}[v_{j-1}]
~=~ \sum_{j=1}^{N^{\prime}-1} \bfM^{j}[{\widetilde{v}}_{j}]
\end{equation}
where each term $\bfM^{j}[{\widetilde{v}}_{j}]$ is in $\PH_{N^{\prime},\alpha^{\prime}}^{p}(\ball)$, with $\widetilde{v}_{j}:=v_{j-1}$ satisfying
$\bfL_{N^{\prime}-j-1}[\widetilde{v}_{j}]=0$.
Rewrite \eqref{eqn:eq7.1inBH} as
\[
0 ~=~ -w + \bfM[\widetilde{v}_{1}] + \cdots + \bfM^{N-1}[\widetilde{v}_{N-1}].
\]
From the uniqueness of the decomposition in Theorem \ref{thm:cellular}, we see that
this is only possible if $w=0$. This proves the claim,
and the proof of Theorem \ref{thm:main2} is complete.
\end{proof}

\section{Proof of Theorem \ref{thm:main3}}

Again, we analyze each term in the cellular decomposition separately. We begin with
the following proposition.

\begin{proposition} \label{prop:Prop7.2inBH}
%\marginpar{Corresponds to Prop 7.2 in BH}
Suppose that $0<p<\infty$, $N\in \bbN$ and $j\in \{0,1,\ldots,N-1\}$. If
$\alpha >a_{N-j,N}(p)$ then there exists a nontrivial $u\in \PH_{N,\alpha}^{p}(\ball)$ of the form
$u=\bfM^{j}[w]$, with $w$ satisfying $\bfL_{N-j-1}[w]=0$.
\end{proposition}

\begin{proof}
When $0<p<1$, we consider the function $u=\bfM^{j}[P_{N-j-1}(\cdot,e_1)]$, where $P_{\theta}$ is
given by \eqref{eqn:poisson}. Explicitly,
\[
u(x) ~=~ U_{N-j,N}(x) ~=~ \frac {(1-|x|^2)^{2N-j-1}}{|x-e_1|^{n+2(N-j-1)}}, \qquad x\in \ball.
\]
By Lemma \ref{lem:testtriviality}, $u$ is in $\PH_{N,\alpha}^{p}(\ball)$ if and only
if $\alpha>b_{N-j,N}(p)$ . Note that $a_{N-j,N}(p)=b_{N-j,N}(p)$ for
$0<p<1$. Hence there exists a nontrivial $u\in \PH_{N,\alpha}^{p}(\ball)$ of the form
$u=\bfM^{j}[w]$, with $w$ satisfying $\bfL_{N-j-1}[w]=0$, provided $\alpha >a_{N-j,N}(p)$.

When $1\leq p<\infty$, we can consider the function $u=\bfM^{j}[\Phi_{N-j-1}]$, where $\Phi_{\theta}$ is defined
by \eqref{eqn:Phi_theta}.
By Corollary \ref{cor:PhiinPH}, $\bfM^{j}[\Phi_{N-j-1}]$ is
in $\PH_{N,\alpha}^{p}(\ball)$ for any $\alpha >-1-jp$. In view of that $a_{N-j,N}(p) = -1-jp$,
this completes the proof.
\end{proof}

\begin{proof}[Proof of Theorem \ref{thm:main3}]
It is a matter of checking which terms actually occur
in the decomposition of Theorem \ref{thm:cellular}.
This is easy to do using Propositions \ref{lem:prop71inBH} and \ref{prop:Prop7.2inBH}.
\end{proof}

\section{Concluding remarks}

We conclude this paper with several remarks and problems which naturally arise from our results.

\begin{problem}
Find an explicit formula for the critical integrability type $\beta(N,p)$ in the range $0<p<\frac {n-2}{n-1}$.
\end{problem}

It is natural to expect that the formula \eqref{eqn:crit-type} in Theorem \ref{thm:main2} is still valid for the
range of $0<p<\frac {n-2}{n-1}$. Nevertheless, it turns out that this is not true even in the simplest case $N=1$.
According to Aleksandrov \cite[p. 526, Remark]{Ale96}, if $n\geq 3$ and $0<p<\frac{n-2}{n}$,
then there exists an $\varepsilon_0 = \varepsilon_0(p)>0$ and a nonzero harmonic function $v$ such that
\[
M_{p}(v,r)= o\left( (1-r)^{1+ 2\varepsilon/p}\right)\quad (\text{as } r\to1)
\]
for $0<\varepsilon<\varepsilon_0$, where
\[
M_{p}(v,r) := \Bigg\{ \int\limits_{\sphere} |v(r\zeta)|^p d\sigma(\zeta) \Bigg\}^{1/p}.
\]
It follows that
\begin{align*}
\left\|\bfM^{N-1}[v]\right\|_{p,-1-Np-\varepsilon}^{p}
~=~& \omega_{n-1} \int_{0}^{1} M_{p}^{p}(v,r) (1-r^2)^{-1-p-\varepsilon} r^{n-1} dr\\
\lesssim~& \int_{0}^{1} (1-r^2)^{-1+\varepsilon} dr ~<~ +\infty,
\end{align*}
which means that
\begin{equation}\label{eqn:MN-1v}
\bfM^{N-1}[v] ~\in~ \PH_{N,-1-Np-\varepsilon}^p (\ball) \quad \text{for all }\, \varepsilon \in (0,\varepsilon_0).
\end{equation}
In particular, when $N=1$, this implies that
\[
\beta(1,p) < -1-p-\varepsilon
\]
for sufficiently small $\varepsilon$.
On the other hand, it is easy to check that
\begin{equation*}
\min\{ b_{0,1}(p), b_{1,1}(p)\} = -1-p \quad \text{for }\ 0<p<\tfrac {n-1}{n}.
\end{equation*}
We then see that
\begin{equation*}
\beta(1,p) ~< ~  \min\{ b_{0,1}(p), b_{1,1}(p)\}
\end{equation*}
for $0<p<\frac {n-2}{n}$. We have not been able to solve this problem, and it could be very difficult.

Borichev and Hedenmalm \cite{BH14} also found an interesting entanglement phenomenon
in the decomposition \eqref{eqn:celluardecomp}.

\begin{definition}
The entangled region $\mathcal{E}_N$ is defined to be the set of $(p,\alpha)$ such that the space $\PH_{N,\alpha}^p(\ball)$ contains
no nontrivial functions of the form $\bfM^{N-1}[v]$ with $v$ harmonic. The complement $\mathcal{N}_N := \mathcal{A}_N \setminus \mathcal{E}_N$
is referred as to the unentangled region.
\end{definition}

Note that we have reformulated the definition of the entangled region $\mathcal{E}_N$ in \cite[Section 3.3]{BH14}, for ease of exposition.
It was shown in \cite[Proposition 3.6]{BH14} that, when $n=2$,
\begin{equation}\label{eqn:E_N2}
\mathcal{E}_N ~=~ \left\{  (p,\alpha)\in \mathcal{A}_N: 0<p< \tfrac {1}{3} \text{ and } \alpha\leq -1-Np \right\}.
\end{equation}

\begin{problem}
Describe the entangled region $\mathcal{E}_N$ when $n\geq 3$.
\end{problem}

When $n\geq 3$, in view of \eqref{eqn:E_N2}, one may conjecture that
\[
\mathcal{E}_N ~=~ \left\{  (p,\alpha)\in \mathcal{A}_N: 0<p< \tfrac{n-1}{n+1} \text{ and } \alpha\leq -1-Np \right\}.
\]
However, by \eqref{eqn:MN-1v}, we see that, for each $0<p<\frac {n-2}{n}$ there exists an $\varepsilon_0 = \varepsilon_0(p)>0$
such that the space $\PH_{N,\alpha}^p(\ball)$ contains a nontrivial functions of the form $\bfM^{N-1}[v]$ with $v$ harmonic,
whenever $\alpha > -1-Np-\varepsilon(p)$. This means that $\mathcal{E}_N$ excludes the region
\[
\left\{  (p,\alpha)\in \mathcal{A}_N: 0<p< \tfrac{n-2}{n} \text{ and } -1-Np-\varepsilon_0(p)
< \alpha <  -1-Np \right\}.
\]
It seems to us that the situation in the higher dimensional case $n\geq 3$ is rather complicated.

We put
\[
\mathcal{N}_N^{(1)} ~:=~ \left\{ (p,\alpha)\in \mathcal{N}_N: u \in \PH_{N,\alpha}^p(\ball) \; \Longrightarrow \; u= \bfM^{N-1}[v]
\text{ for some harmonic } v \right\}
\]
and refer to it as the principal unentangled cell.
Next result is a higher-dimensional extension of \cite[Proposition 3.7]{BH14}.

\begin{proposition} \label{prop:main4}
Let $p \geq \frac {n-2}{n-1}$ and $N\in \bbN$ be fixed.
Every $u\in \PH_{N,\alpha}^p(\ball)$ has the form $u=\bfM^{N-1}[v]$ with $v$ harmonic on $\ball$ if and only if
\[
\alpha ~\leq~ \min\{-n-(N-n-1)p,\, -1-(N-2)p\}.
\]
In other words,
\[
\textstyle{\mathcal{N}_N^{(1)}  \bigcap \widetilde{\mathcal{A}}_N}  = \left\{ (p,\alpha): p\geq \tfrac {n-2}{n-1} \text{ and }
\alpha ~\leq~ \min\{-n-(N-n-1)p,\, -1-(N-2)p\} \right\}.
\]

\end{proposition}

\begin{proof}[Proof of Proposition \ref{prop:main4}]
In terms of the decomposition in Theorem \ref{thm:cellular}, it is
a matter of deciding for which $(p,\alpha)$ the functions $w_j$, with $j = 0, \ldots, N-2$, must all
equal $0$. This can be done by using Propositions \ref{lem:prop71inBH} and \ref{prop:Prop7.2inBH}.
\end{proof}

\bibliographystyle{amsplain}

\end{document}